\documentclass[11pt,reqno]{amsart}
\usepackage{mathrsfs}  
\usepackage{amscd,amsmath,amssymb,amsfonts,verbatim,graphicx,enumitem}
\usepackage[colorlinks]{hyperref}
\usepackage{tikz-cd}
\usepackage{graphicx}
\usepackage{float}

\newcommand{\Ric}{\mathrm{Ric} \,}

\newcommand{\R}{\mathbb{R}}

\newcommand{\ddbar}{i\partial\overline{\partial}}

\newcommand{\Supp}{\mathrm {Supp}}
\renewcommand{\epsilon}{\varepsilon}
\renewcommand{\leq}{\leqslant}
\renewcommand{\geq}{\geqslant}
\newcommand{\Dlc}{D_{\rm lc}}
\newcommand{\Dklt}{D_{\rm klt}}

\def\XXint#1#2#3{{\setbox0=\hbox{$#1{#2#3}{\int}$ }
\vcenter{\hbox{$#2#3$ }}\kern-.6\wd0}}

\newtheorem{theorem}{Theorem}[section]

\newtheorem{lemma}[theorem]{Lemma}
\newtheorem{corollary}[theorem]{Corollary}
\newtheorem{proposition}[theorem]{Proposition}

\theoremstyle{definition}
\newtheorem{definition}[theorem]{Definition}
\newtheorem{remark}[theorem]{Remark}

\numberwithin{equation}{section}

\begin{document}
	
\title{Variation of K\"ahler-Einstein metrics with mixed singularities}
	
\begin{abstract} In this short note, we consider a fibration $\pi: (\mathcal X,\Delta)\to Y$ between two compact K\"ahler manifolds with generic fiber of $\pi$ being a smooth log canonical pair with ample canonical divisor, we prove that the current induced by variation of K\"ahler-Einsteins with mixed cone and Poincare singularities is positive. This generalizes the result of Schumacher in the smooth case \cite{Schu} and the result of Guenancia in the conic case \cite{G}. As a consequence, we prove a new positivity result for the relative canonical bundle. Using this positivity result, we show that if $-(K_{X}+\Delta)$ is nef, then $-K_Y$ is pseudo-effective, which partially answers a question of Demailly-Peternell-Schneider in the K\"ahler case and moreover, the Albanese map associated to $X$ is surjective.

\end{abstract}
	
\author{Xin Fu}
\address{School of sciences, Institute for Theoretic Sciences, Westlake University, Hangzhou, Zhejiang Province, 310030, China}

\email{fuxin54@westlake.edu.cn}
	\author{Jiyuan Han}
\address{School of sciences, Institute for Theoretic Sciences, Westlake University, Hangzhou, Zhejiang Province, 310030, China}

\email{hanjiyuan@westlake.edu.cn}

\author{Yongpan Zou}
\address{School of sciences, Institute for Theoretic Sciences, Westlake University, Hangzhou, Zhejiang Province, 310030, China}

\email{zouyongpan@westlake.edu.cn}
	
\date{\today}
	
\maketitle
	
	
\setcounter{tocdepth}{3}
\section{Introduction}
Given a fibration $\pi: \mathcal X\to Y$ between two K\"ahler manifolds, it is important to understand the property of the relative canonical bundle $K_{\mathcal X/Y}$. A differential algebraic approach is to equip the relative canonical bundle with some canonical metrics. One useful metric is the so called Bergman metric (c.f \cite{BP}), which works in very general context. On the other hand, when the generic fiber of $\pi$ is canonical polarized or Calabi-Yau, another useful tool is to study the variation of K\"ahler-Einstein metrics associated to this family and this leads to deep application to the positive of property of $K_{\mathcal X/Y}$ and related moduli problems (c.f. \cite{Tian,Schu,Tsu,CGP1,CGP2,G,BCS,Nau}).

Going back to this short note, we  we consider a fibration $\pi: \mathcal X\to Y$ between two K\"ahler manifolds and moreover we assume the generic fiber of $\pi$ is a canonical polarized smooth log canonical pair. This set up is well studied by \cite{Schu,Tsu} when the generic fiber is a smooth canonical polarized manifold and by \cite{G} when the generic fiber is a smooth canonical polarized Klt pair $(X,D)$. In short, all of them are able to show that the variation of K\"ahler-Einstein metric will induce a positive singular metric on $K_{\mathcal X/Y}+D$. In \cite{G}, Guenancia asks the question that how about the case when the generic fiber is a smooth log pair. \cite{Nau,To} is able to confirm the positivity of $K_{\mathcal X/Y}+D$ when the boundary component $D$ is reduced.  In this note, we are able to prove the positivity of $K_{\mathcal X/Y}+D$ only requiring that all the components of $D$ have coefficient lies in $[0,1]$ by using a different argument. So it seems that our more general result is new. Now let us state our main result precisely.
\begin{theorem}\label{t1}
Let $\pi:\mathcal X\to Y$ a holomorphic surjective map between compact K\"ahler manifolds, $D=\sum_{i=1}^r(1-\beta_i) D_i$, where $D_i$ is a reduced divisor with generically simple normal crossings and   $\beta_i\in [0,1]$ such that the cohomology class $c_1(K_{\mathcal X_y}+\sum_{i=1}^r (1-\beta_i){D_i}_{|{\mathcal X_y}})+\{\gamma\}$ is K\"ahler for every  $y\in Y\setminus\mathcal{S} $, where $\mathcal S$ is the singular locus of $\pi$ and $\gamma$ is a smooth semi-positive $1$-$1$ form on $X$. Then by gluing the fiberwise K\"ahler-Einstein metric with mixed singularity, we obtain a closed positive current $\rho_{mix}$ in 
$ c_1(K_{\mathcal X/Y}+ \sum_{i=1}^r(1-\beta_i) D_i)+\{\gamma\}$ (See section \eqref{e:rhodef} for the precise definition of $\rho_{mix}$).
\end{theorem}

Now we explain the idea of proof. We assume that $\gamma=0$. Under the same set-up (but require $\beta_i> 0$), by gluing fiberwise conic K\"ah le r-Einstein metrics, Guenancia \cite{G} obtains a closed positive current $\rho$ in $ c_1(K_{\mathcal X/Y}+ \sum_{i=1}^r(1-\beta_i) D_i)$. Now if $\beta_1=0$ (for simplicity, we assume  $0<\beta_i<1,i>1$), we perturb $D_1$ by any sufficiently small number $\delta>0$ to $(1-\delta)D_1$ , then for such $\delta$, there exists a closed positive current $\rho_\delta$ by Guenanica's result. Then the key is to show  the weak limit of $\rho_\delta$, when $\delta\to 0$, is the desired current $\rho_{mix}$ by gluing fiberwise K\"ahler-Einstein metric with mixed singularities. 
On a fixed fiber, such a convergence result is proved by Guenancia \cite{G2}. We prove the weak convergence of $\rho_\delta$ to $\rho_{mix}$ on the locus  where the fibration $\pi$ is smooth.

As a consequence of Theorem \ref{t1}, we obtain a new positivity result for the relative canonical bundle. 
\begin{theorem} \label{psef}
Let $p \colon X \to Y$ be a surjective map between two compact K\"{a}hler manifolds. 
Let \( F_{\alpha}:= \sum \alpha_i F_i \) be an \( \mathbb{R} \)-divisor whose supported divisor is as in the above setup, where each \( \alpha_i \in [0,1] \).
If $L \to X$ is a nef \(\mathbb{R} \)-line bundle, such that the adjoint system $K_{X_y} + (L+F_{\alpha})|_{X_y} $ is nef for generic $y \in Y$. Then the \( \mathbb{R} \)-bundle $K_{X/Y} + L +F_{\alpha}$ is pseudo-effective.
\end{theorem}

Since the canonical positively curved singular metric for an irreducible divisor $D$ is not $L^2$ integrable, it will be interesting to see if the $L^2$ method can be used to obtain the above result. On the other hand, the anonymous referee points out that the above theorem can be proved by using Guenancia's positivity result (c.f. Theorem \ref{t:positivity} and Remark \ref{newproof}).

We give two applications of the positivity theorem above. For the first one, we address a question of Demailly-Peternell-Schneider: Let $\pi$ be a surjective morphism from a log canonical pair $(X,D)$ onto a $\mathbb Q$-Gorenstein variety $Y$. If $-(K_X+D)$ is nef, then $-K_Y$ is  pseudo-effective. In the projective case, this conjecture is fully confirmed by Chen-Zhang in \cite{CZ}. We partially confirm this conjecture in the K\"ahler case, more precisely, we have
\begin{theorem}\label{Demailly} Suppose $\pi:(X,D)\to Y$ is a fibration between two compact K\"ahler manifolds. Further suppose that $(X,D)$ is log canonical and $-(K_X+D)$ is nef, then $-K_Y$ is pseudo-effective. 
\end{theorem}
For the second one, we prove that given a smooth log canonical pair $(X,D)$ with nef anti canonical divisor, the Albanese map is surjective, following the idea of Paun \cite{Paun}. Let us briefly recall known results on the surjectivity of Albanese map when the anti canonical bundle is nef.  In \cite{DPS}, the authors conjectured that the Albanese map \(\alpha_{X}\) is surjective when $X$ is a smooth K\"ahler manifold when \(-K_{X}\) is nef. This conjectured is proved by Zhang \cite{Zha96,Zha19} when $X$ is a projective KLT variety and proved by M. P\u{a}un  \cite{Paun} when $X$ is smooth and K\"ahler. In a recent preprint \cite{MWWZ}, the authors are able to confirm the surjectivity of Albanese map when $(X,D)$ is a K\"ahler KLT pair ($X$ smooth and $D$ not necessary smooth). 
Here, we prove the surjectivity of Albanese map for a smooth log canonical pair. As far as we know, this is new.
\begin{theorem} \label{t2}
Let \(X\) be a compact K\"ahler manifold, and let $D=\sum D_i$ be a simple normal crossing divisor on \(X\), and log canonical divisor \(D_{\alpha} = \sum \alpha_i D_i\) with coefficients \(\alpha_i \in [0,1]\). Assume that the \(\mathbb{R}\)-line bundle \(-(K_{X} + D_{\alpha})\) is nef. Then the Albanese morphism \(\alpha_{X}\colon X \to \mathrm{Alb}(X)\) is surjective.
\end{theorem}
\thanks{\textbf{Acknowledgement:} We would like to thank Professor Henri Guenancia, Jiawei Liu for helpful discussions. We also thank Professor Gang Tian for his suggestion to include a more general case of our main Theorem.  We also thank the anonymous referee for useful comments.  Xin Fu is  supported by National Key R\&D Program of China  2024YFA1014800 and NSFC No. 12401073. Yongpan Zou is supported by National Key R\&D Program of China 2023YFA1009900.}
\section{Construction of K\"ahler-Einstein metrics} In this section, we recall the known results on construction of K\"ahler-Einstein metrics with negative scalar curvature. Firstly we introduce a  smooth log canonical pair.
	\begin{definition}A smooth \textit{log canonical} pairs $(X,D)$ consists of  a compact Kähler manifold $X$ and a divisor $D= \sum_{i=1}^{s} D_i +\sum_{i=s+1}^r (1-\beta_i)D_i $ having simple normal crossing support and coefficients $\beta_i \in (0,1)$. We write $X_0:=X \setminus \Supp (D)$, $D= \Dklt+\Dlc$ where  $\Dlc=\sum_{i=1}^s D_i$ and $\Dklt:= \sum_{i=s+1}^r (1-\beta_i)D_i$.
 \end{definition}

\subsection{Mixed cone and cusp singularities} In this subsection, we recall how K\"ahler-Einstein metric with mixed conic and cusp singularties are constructed on a log smooth log canonical pair $(X,D)$ with $K_X+D>0$ ample. When $D=\emptyset$, Aubin \cite{Aub} and Yau \cite{Yau} confirmed the existence of K\"ahler-Einstein metric on $X$.
When $\Dklt=0$, it was showed by Kobayashi \cite{Kob} and Tian-Yau \cite{TY} that whenever $K_X+D$ is ample, there exists a unique negatively curved K\"ahler-Einstein metric on $X_0$ having cusp  singularities along $D$. More generally, when both $\Dklt$ and $\Dlc$ are none empty, Guenanica-Paun \cite{GP} (see also \cite{CGP,Don,Bre,JMR,DS} for related results) constructed a unique K\"ahler-Einstein metric with mixed singularity associted to the pair $(X,D)$. These metrics are smooth K\"ahler-Einstein metrics on $X_0$, and have mixed cone and cusp singularities along $D$, i.e., being locally quasi-isometric to the model metric
\[\omega_{\rm mod}:=\sum_{i=1}^s \frac{i dz_i\wedge d\bar z_i}{|z_i|^{2} \log^{2}|z_i|^{2}}+\sum_{i=s+1}^r \frac{i dz_i\wedge d\bar z_i}{|z_i|^{2(1-\beta_i)}}+\sum_{j=r+1}^n i dz_i\wedge d\bar z_i \] 
if $(X_0,D)$ is locally isomorphic to $(X_{\rm mod}, D_{\rm mod})$, where $X_{\rm mod}=(\mathbb{D}^*)^s\times (\mathbb{D}^*)^{r-s} \times \mathbb{D}^{n-r}$, $D_{\rm mod}=[z_{1}=0]+\cdots+(1-\beta_{s+1}) [z_{s+1}=0]+\cdots+(1-\beta_r) [z_r=0]+\cdots + [z_{n}=0]$; $\mathbb{D}$ (resp. $\mathbb{D}^*$) being the disc (resp. punctured disc) of radius $1/2$ in $\mathbb C$.

Let $\omega$ be a smooth K\"ahler metric on $X$, $L_i$ be the line bundle associated to divisor $D_i$, $h_i$ be a smooth Hermitian metric on $L_i$. Also let $\theta_i:=\Ric(h_i)$ be the curvature form of metric $h_i$, and set \[\theta:=\sum_{i=1}^{r}(1-\beta_i)\theta_i, \,\,\tilde\omega=-\Ric(\omega^n)+\theta.\]
Consider the following Monge-Amp\`ere equation related to K\"ahler-Einstein metric with mixed singularity:

\begin{equation}\label{e:mixMA} (\tilde\omega+\ddbar\phi_{mix})^n=\frac{e^{\phi_{mix}}\omega^n}{\prod_{i=1}^s|s_i|_{h_i}^{2}\prod_{i=s+1}^r|s_i|_{h_i}^{2-2\beta_i}},\end{equation}
where  $s_i$ is the defining section of line bundle $L_i$ associated to divisor $D_i$. We have the following existence result \cite{GP}.
\begin{theorem}[Guenancia-Paun]\label{t:GP-existence}
\label{thm:lc}
Let $(X,D)$ be a log smooth log canonical pair such that $K_X+D$ is ample. Then there exists a solution $\phi_{mix}$ to \eqref{e:mixMA} satisfying: 
\begin{enumerate}
\item $\phi_{mix}$ is smooth on $X_0$,
\item $-\sum_{i=s+1}^r\ddbar\log\log^2\prod_{i=s+1}^r|s_i|_{h_i}-C\leq\phi_{mix}\leq C$ for some fixed constant $C$,
\item  Set $\omega_{{mix}}:=\tilde\omega+\phi_{mix}$, then $\Ric\omega_{mix} = -\omega_{mix}$ on $X_0$,
\item  $\omega_{mix}$ has mixed cone and cusp singularities along $D$.
\end{enumerate}
\end{theorem}
\begin{remark}\label{twisted}
It is clear that if we assume that $\gamma$ is a smooth semi-positive form on $X$ and $K_X+D+\{\gamma\}$ is K\"ahler, we may obtain the twisted K\"ahler-Einstein metric by solving the following equation:
\begin{equation}\label{e:twisted} (\tilde\omega+\gamma+\ddbar\phi_{mix})^n=\frac{e^{\phi_{mix}}\omega^n}{\prod_{i=1}^s|s_i|_{h_i}^{2}\prod_{i=s+1}^r|s_i|_{h_i}^{2-2\beta_i}}.\end{equation}
\end{remark}
For later purpose, we do not assume $\tilde\omega$ is  K\"ahler here. However, by assumption, the class $[\tilde\omega]$ is K\"ahler, we may still solve the Monge-Ampere equation \eqref{e:mixMA}. Also our equation \eqref{e:mixMA} related to K\"ahler-Einstein metric  is slightly different from the equation in \cite{GP}  since we use different reference form $\tilde\omega$. In \cite{GP}, they used the Poincare type reference metric $\tilde\omega+\sum_{i=s+1}^r\ddbar\log\log^2\prod_{i=s+1}^r|s_i|_{h_i}$.
\subsection{Guenancia's positivity result} In this subsection,  we recall the result of Guenancia \cite{G} which in turn developed the results of \cite{Schu} (see also \cite {Tsu}) for the absolute case (no boundary divisor $D$). 

Now let us set-up the question more precisely.
Let $\pi:\mathcal X \to Y$ be a holomorphic surjective map between two compact K\"ahler manifolds $\mathcal X$ and $Y$, and let $D=\sum_{i=1}^r D_i$ be a reduced divisor on $\mathcal X$ with generically simple normal crossings and mapping surjectively to $Y$ by $\pi$. We denote by $\mathcal{S} \subset Y$ the minimal analytic subset of $Y$ such that if $\mathcal X_0:=\pi^{-1}(Y\setminus\mathcal{S})$, then every fiber $\mathcal X_y$ of $\pi_{|\mathcal X_0}$ is smooth, $D_{|\mathcal X_y}$ has simple normal crossings (and therefore is transverse to $\mathcal X_y$).  

Then we explain how  a smooth $(1,1)$-form $\omega$ on $\mathcal X$ can induce a singular metric on the bundle
$K_{\mathcal X/Y}$ by following \cite{Paun}. Fix a smooth $(1,1)$-form $\omega$ on $\mathcal X$, whose restriction to the
fibers of $\pi$ is K\"ahler. Then $\omega$ induces a metric on the bundle
$K_{\mathcal X/Y}$ as follows: Let $x\in \mathcal X$ be a point, and let $U$ be a coordinate set of $\mathcal X$ centered at
$x$. We denote by $z_1
,\cdots, z_{n+d}$ a coordinate system on $U$, and we equally
introduce $t_1
,\cdots, t_d$
coordinates near the point $y=p(x)$. This data
induces a trivialization of the relative canonical bundle, with respect
to which the weight of the metric we want to introduce is given by the
function $\Phi_U$, defined by the equality
\[\omega^n\prod_{j=1}^n\sqrt{-1}\pi^*(dt\wedge d\bar t_j)=e^{\Phi_U}\prod_{j=1}^{n+d}\sqrt{-1}(dz_j\wedge d\bar z_j)\]

The functions $\Phi_U$ glue together as weights of a globally
defined metric denoted by 
$h^\omega_{\mathcal X/Y}$ on the relative canonical bundle, the
corresponding curvature form is simply $\ddbar\Phi_U$. Since the weight $\Phi_U$ may have a log pole when the Jacobian of $\pi$ has zero,  the resulting metric $h^\omega_{\mathcal X/Y}$ will be identically $+\infty$ and hence $\ddbar\Phi_U$ is globally defined on $\mathcal X$ as a current.

Assume that for a generic $y\in Y$ and a set of numbers $\beta_1, \ldots, \beta_r \in [0,1)$, the cohomology class 
$$c_1(K_{\mathcal X_y}+\sum_{i=1}^r (1-\beta_i){D_i}_{|{\mathcal X_y}})+\gamma_y$$
is K\"ahler. Set $$\tilde\omega:=Ric(h^\omega_{\mathcal X/Y})+\sum_{i=1}^r (1-\beta_i){\theta_i}_{|{\mathcal X_y}}+\gamma_y.$$By Theorem \ref{t:GP-existence} and Remark \ref{twisted}, there exists on each such fiber $\mathcal X_y$ a unique (twisted)  K\"ahler-Einstein metric $$\rho_y:=\tilde\omega_y+\ddbar\phi_y$$ with mixed singularity along ${D_i}_{|\mathcal X_y}$ satisfying:
\begin{equation}
\label{KE}
\Ric \rho_{y} = -\rho_{y}+ \sum_{i=1}^r(1-\beta_i) [{D_i}_{|\mathcal X_y}]+\gamma_y,
\end{equation}
over $\mathcal X_y$ by solving Monge-Ampere equation \eqref{e:mixMA}.

To study the variation property of the  K\"ahler-Einstein metrics, we  glue the fiberwise  metrics $\rho_{y}$ to get a current $\rho\in c_1(K_{\mathcal X/Y}+ \sum_{i=1}^r(1-\beta_i) D_i)+\{\gamma\}$ with locally bounded potentials, i.e.
\begin{equation}\label{e:rhodef}
\rho:=\tilde\omega+\ddbar\phi,
\end{equation}
where $\ddbar$ is taken over $\mathcal X$. However, we remark that without any regularity of $\phi$ been obtained,  we do not justify $\ddbar\phi$ is well defined even on $\mathcal X_0:=\pi^{-1}(Y\setminus \mathcal S)$. In the end, we will see that $\phi$ is a quasi PSH function and we do not obtain any higher regularity of $\phi$ although some mild improvement is possible.

In the conic set-up (all $\beta_i>0$), Guenancia \cite{G} confirmed that $\rho$ is positive on $\mathcal X_0$ and can be extended to $\mathcal X$ as a positive current.

\vspace{4mm}
\begin{theorem}[Guenancia]\label{t:positivity}
Let $\pi:\mathcal X\to Y$ a holomorphic surjective map between compact Kähler manifolds, $D=\sum_{k=1}^r D_k$ a reduced divisor with generically simple normal crossings, $\{\gamma\}\in H^{1,1}(\mathcal X,\R)$ a semipositive class and $\beta_1, \ldots, \beta_r \in (0,1)$ such that the cohomology class $c_1(K_{\mathcal X_y}+\sum_{i=1}^r (1-\beta_i){D_i}_{|{\mathcal X_y}})+\{\gamma\}$ is K\"ahler for every  $y\in Y\setminus\mathcal{S} $. Then the following holds:
\begin{enumerate}
\item $\rho$ is positive
\item  $\rho$ is bounded outside $D$
\item  $\rho$ extends to $\mathcal X$ as a closed positive current in $ c_1(K_{\mathcal X/Y}+ \sum_{i=1}^r(1-\beta_i) D_i)+\{\gamma\}$.
\end{enumerate}
In particular, the cohomology class $ c_1(K_{\mathcal X/Y}+ \sum_{i=1}^r(1-\beta_i) D_i)+\{\gamma\}$  is pseudoeffective. 
\end{theorem}
\vspace{4mm}

\section{Cone to cusp convergence} In this section, we recall the result of \cite{G2}. As mentioned in the introduction, the main idea to prove Theorem \ref{t1} is investigating the metric convergence in suitable sense when cone angle goes to zero. Such kind of convergence on a fixed smoot fiber is  studied by Guenancia in \cite{G2}, which we recall below. (See also an interesting K\"ahler Ricci flow analogue in \cite{LZ}.)

We fix a log smooth log canonical pair $(X,D=\sum_{i=1}^{s} D_i+\sum_{i=s+1}^{r}(1-\beta_i)D_i)$ and a smooth semi-positive form $\gamma$ such that $K_X+D+\{\gamma\}>0$. Then for small  $0<\delta_i<\delta_0,1\leq i\leq s$, where $\delta_0$ is small constant, we have \begin{equation}K_X+\sum_{i=1}^{s}(1-\delta_i)D_i+\sum_{i=s+1}^{r}(1-\beta_i)D_i+\{\gamma\}>0.\end{equation} 

\noindent Fix a set of positive constants $\delta=(\delta_1,\delta_2,...\delta_s)>0$, by \cite{GP}, there are twisted conic K\"ahler-Einstein metrics $\rho_{\delta}$ with  cone angle $2\pi \delta_i$ along $D_i, 1\leq i\leq s$ and cone angle $2\pi\beta_i$ along $D_i, s<i\leq r$ by solving the following equation:
\begin{equation}\label{e:appro} (-\Ric(\omega)+\Ric(D)-\delta\theta+\gamma+\ddbar\phi_\delta)^n=\frac{e^{\phi_\delta}\omega^n}{\prod_{i=1}^s|s_i|_{h_i}^{2-2\delta_i}\prod_{i=s+1}^r|s_i|_{h_i}^{2-2\beta_i}},\end{equation}
where $\delta\theta:=\sum_{i=1}^{s}\delta_i\theta_i$ and $\omega$ is a K\"ahler metric on $X$.

\noindent In local coordinate, $\rho_{\delta}$ is quasi-isometric to the cone metric:
\[\omega_{\rm mod} =\sum_{i=1}^{s}\frac{i dz_i \wedge d \bar z_i}{|z_i|^{2(1-\delta_i)} }+ \sum_{i=s+1}^{r}\frac{i dz_i \wedge d \bar z_i}{|z_i|^{2(1-\beta_i)} } + \sum_{k=1}^n i dz_k \wedge d \bar z_k\]

\noindent So we have a family of metrics $(\rho_\delta)_{0< \delta<\delta_0}$ on $X\setminus D$  satisfying the twisted K\"ahler-Einstein equation:
\[\Ric \rho_\delta = - \rho_\delta + \sum_{i=s+1}^r(1-\beta_i) [D_i]+\sum_{i=1}^s(1-\delta_i)[D_i]+\gamma.\]

\noindent  Guenancia proves the following convergence result for $\rho_\delta$, when $\delta\rightarrow 0$.
\begin{theorem}[Guenancia]\label{t:conetocusp}Consider equations \eqref{e:appro}. When $\delta\to 0$, 
the twisetd conic K\"ahler-Einstein metric $\rho_\delta:=-\Ric(\omega)+\Ric(D)+\gamma-\delta\theta+\ddbar\phi_\delta$  converge to the twisted K\"ahler-Einstein metric $\rho_{mix}$ with mixed singularities, both in the weak topology of currents and in the $\mathscr C_{\rm loc}^{\infty}(X\smallsetminus D)$ topology. Moreover, we have 
\begin{equation}\label{c0}-\sum_{i=1}^r\log\log^2|s_i|_{h_i}-C\leq\phi_\delta\leq C \end{equation} for some $C$ independent of $\delta$.
\end{theorem}
Note that in \cite{G2}, the reference form is K\"ahler and here $-\Ric(\omega)+\Ric(D)+\{\gamma\}$ is not necessary  K\"ahler but only in a K\"ahler class. So the solutions associated with different reference forms will differ by a smooth function on $X$, which is independent of $\delta$. 

Now we consider a smooth fibration and obtain  suitable uniform control of $\phi_\delta$ also in the base variable $t\in Y$. The following lemma is a generalization of  \cite[Lemma 2.1]{G2}. 
\begin{lemma}\label{uni} Recall that the fibration $\pi:(\mathcal X_0,D)\to Y_0$ is smooth when restricted on $Y_0:= Y\setminus \mathcal S$.  Fix a compact set $K\subset Y_0$, then there is a constant $C$ independent of both $t\in K$ and $\delta<\delta_0$, such that 
\begin{equation}\label{c00}-\sum_{i=1}^r\log\log^2|s_i|_{h_i}-C\leq\phi_{y,\delta}\leq C, \end{equation} 
where $\phi_{y,\delta}$ is the solution to \eqref{e:appro} on fiber $\mathcal X_y.$
\end{lemma}
\begin{proof}We sketch the proof. Fix a point $y\in K$, by compactness, it is clear that it suffices to prove the desired estimates in a neighbourhood of $y$. Then we pick a smooth function $f$ on $\mathcal X_{y}$ such that $$-\Ric(\omega)+\Ric(D)+\gamma+\ddbar f>0 \quad \textnormal{on}\,\,\, \mathcal X_{y}.$$
Do an arbitrary smooth extension of $f$ to $\mathcal X_0$, which is still denoted by $f$. Note that $-\Ric(\omega)+\Ric(D)+\gamma$ is defined in a neighbourhood of  $\mathcal X_{y}$, so by continuity it is clear that  $-\Ric(\omega)+\Ric(D)+\gamma+\ddbar f>0$ on $\pi^{-1}(U)$, where $U$ is a small neighbourhood of $y$. Since $f$ is bounded on $\pi^{-1}(U)$, without lose of generality, we may assume that the reference form $-\Ric(\omega)+\Ric(D)+\gamma-\delta\theta$ in \eqref{e:appro} is K\"ahler, for sufficiently small $\delta$,  by forgetting $f$.

 To release the notation, we also omit the parameter $y$ in $\phi_{y,\delta}$ in the rest of proof. Now the proof follows from standard maximum principle argument. We have the following upper bound $$\phi_\delta\leq\log\frac{(-\Ric(\omega)+\Ric(D)+\gamma-\delta\theta)^n}{\prod_{i=1}^s|s_i|_{h_i}^{2\delta_i-2}\prod_{i=s+1}^r|s_i|_{h_i}^{2\beta_i-2}\omega^n}\leq C.$$
Now we derive the lower bound:
\begin{equation}\phi_\delta+\sum_{i=1}^r\log\log^2|s_i|^2_{h_i}\geq\log\frac{(-\Ric(\omega)+\Ric(D)+\gamma-\delta\theta-\ddbar\sum_{i=1}^r\log\log^2|s_i|^2_{h_i})^n\prod_{i=1}^r\log^2|s_i|^2_{h_i}}{\prod_{i=1}^s|s_i|_{h_i}^{2\delta_i-2}\prod_{i=s+1}^r|s_i|_{h_i}^{2\beta_i-2}\omega^n}.\end{equation}
We claim that the right hand term above is uniformly bounded below. For this, without lose of generality we may assume that all $\delta_i,\beta_i$ are zero. Under this assumption, we compute as follows: firstly, set the $1$-$1$ form on $\mathcal X_0$  $$\bar\omega:=-\Ric(\omega)+\Ric(D)+\gamma-\delta\theta,$$  whose restriction to a fiber is K\"ahler. By scaling of the metric $h_i$, we may assume that $|s_i|_{h_i}<<1$. Then the following holds on any fiber $\mathcal X_y, $ when $y\in K$ 
\begin{equation}\label{Aa}(\bar\omega-\ddbar\sum_{i=1}^r\log\log^2|s_i|^2_{h_i})^n\geq (\frac{1}{2}\bar\omega+\sum_{i=1}^ri\frac{\partial|s_i|^2_{h_i}\wedge\bar\partial|s_i|^2_{h_i}}{|s_i|^4_{h_i}\log^2|s_i|^2_{h_i}})^n.\end{equation}
Again by compactness, it suffices to show that a positive lower bound exists on a coordinates chart $(t,z_1,...z_n)$, where $t$ is the coordinate on base, and $z_i$ are the coordinates on the fiber. Without lose of generality, we focus on the region where all $|z_i|,i=1...r$ are much smaller than the  smooth metric $h_i$. Then the right hand side of \eqref{Aa} is bounded below as:
$$(\frac{1}{2}\bar\omega)^{n-r}\prod_{i=1}^r\frac{\partial|s_i|^2_{h_i}\wedge\bar\partial|s_i|^2_{h_i}}{|s_i|^4_{h_i}\log^2|s_i|^2_{h_i}}\geq C_K \frac{\omega^n}{\prod_{i=1}^r|s_i|^2_{h_i}\log^2|s_i|^2_{h_i}},$$
where $C_K$ is a constant depending on $K$. The lemma is proved.
\end{proof}
\section{Proof of theorem \ref{t1} and Theorem \ref{psef}}
In this section, we firstly prove Theorem \ref{t1} by using results in previous subsections and Theorem \ref{psef} will follow immediately. Recall that, the fiberwise K\"ahler-Einstein metric $\rho_{y,mix}$ with mixed singularity can be glued to a $(1,1)$- current $\rho_{mix}$ on $\mathcal X_0$. In the following lemma, by using cone deformation argument, we first show that $\rho_{mix}$ is a positive $(1,1)$-current for a smooth fibration $\pi:\mathcal X_0\to Y_0$. 
\begin{lemma}\label{l:conedeformation}$\rho_{mix}$ is positive $(1,1)$-current on $\mathcal X_0$. 
\end{lemma}
\begin{proof}

When $\delta_i>0,i=1,2,..s$, all cone angles $\delta_i$ are nonzero.  By Theorem \ref{t:positivity}, $\rho_{y,\delta_i}$ can be glued to a positive $(1,1)-$ current $\rho_{\delta_i}\geq 0$ on $\mathcal X_0$. 

Now we fixed a point $y\in Y_0$, then $\mathcal X_{y}$ is a smooth compact K\"ahler manifold. On  each fixed fiber $\mathcal X_{y}$, when $\delta_i\to 0$, by Theorem \ref{t:conetocusp},  $\rho_{y,\delta_i}$ converges smoothly to $\rho_{y,mix}$ outside $D|_{\mathcal X_{y}}$. This implies that $\rho_{\delta_i}$ converges pointwisely to $\rho_{mix}$ on $\mathcal X_0\setminus D$. By the uniform local potential estimate of $\rho_{\delta}$ in (\ref{c00})  and dominate convergence theorem, we deduce that the local potential of $\rho_\delta $ converges to the local potential $\rho_{mix}$ on $\mathcal X_0\setminus D$ in $L^1_{loc}$ sense and hence $\rho_{mix}$ is a positive $(1,1)-$ current on $\mathcal X_0$. Again, by the uniform upper bound of the local potential of $\rho_{mix}$ in \eqref{c00} and Hartogs extension theorem, we deduce that $\rho_{mix}$ is indeed positive on $\pi^{-1}(K)$, where $K$ is any compact subset in $Y_0$, and hence also positive on $\mathcal X_0$.
\end{proof}
Now we know that the fiberwise twisted  K\"ahler-Einstein metrics with mixed singularity induce a closed positive $(1,1)$-current  $\rho_{mix} \in c_1(K_{\mathcal X/Y}+\sum(1-\beta_k)D_k)_{|\mathcal X_0}+\{\gamma\}$ on $\mathcal X_0$. To complete the proof of Theorem \ref{t1}, we would like to show that the current $\rho_{mix}$ could be extended to $\mathcal X$ and hence it suffices to prove that the local potential of $\rho_{mix}$
 is bounded from above near the singular fiber. To do this, in the following lemma, we follow Paun's argument in \cite[\S 3.3]{Paun}.

\begin{lemma}\label{OT}$\rho_{mix}$ can be extended across the singular fiber to $\mathcal X$ as a positive current.
\end{lemma}

\begin{proof}
We pick a point $x_0$ in $\mathcal X_0=\pi^{-1}(Y_0)$, and choose a Stein neighborhood $\Omega$ of $x_0$ in $\mathcal X$; we write $\Omega_y = \Omega \cap \mathcal X_y$, choose
a potential $\varphi_y$ of $\rho_y$ so that (up to adding a pluriharmonic function to $\psi_y$) the equation satisfied by $\varphi_y$ on $\Omega_y$ is (c.f. \eqref{e:appro})
\begin{equation}\label{singularKE}(\ddbar \varphi_y)^n = e^{\varphi_{y}-G}\left| \frac {dz}{dt} \right|^2\end{equation}
where $G$ is function on $\Omega_y$ and the coordinates $(z_1, \ldots, z_n, t_1, \ldots, t_m)$ are chosen so that $p(\underline z, \underline t)=\underline t$. By comparing to \eqref{e:appro}, crucially, we claim that $G$ is uniformly bounded above on $\Omega$ and hence on $\mathcal X$. This is due to the fact that for a singular fibration between smooth manifolds, $\frac{dz}{dt}$ will only have zeros. 
We set
$$H_{m,y}:=\left\{f\in \mathcal O(\Omega_y); \int_{\Omega_y}|f|^2e^{-(m-1)\varphi_y}\left| \frac {dz}{dt} \right|^2 \le 1\right\}$$

Note that we used the weight $(m-1)\varphi_y$ instead of $m\varphi$ in the above integral. By Demailly's regularization Theorem \cite{Demailly}, we have
 \begin{equation}\label{Dreg}\varphi(y)(x_0)= \lim_{m\to \infty}\sup_{f\in H_{m,y}} \frac 1 m \log |f(x_0)|\end{equation}

Fix $f\in H_{m,y}$, 
 the $L^{2/m}$ version of Ohsawa-Takegoshi extension theorem \cite{BP2} yields a holomorphic function $F$ on $\Omega$ that extends $f$ and such that
\begin{align*}|F(x_0)|^{2/m} &\le C_{\Omega} \int_{\Omega} |F|^{2/m} |dz|^2 \le C \int_{\Omega_y}|f|^{2/m}\left| \frac {dz}{dt} \right|^2= C \int_{\Omega_y}|f|^{2/m} e^{-\varphi_y+G}(\ddbar \varphi_y)^n,\end{align*}
where in the last equation, we have used the equation \eqref{singularKE}. We further deduce by Holder inequality that
\begin{align*} \int_{\Omega_y}|f|^{2/m} e^{-\varphi_y+G}(\ddbar \varphi_y)^n&\leq C'\int_{\Omega_y}|f|^{2} e^{-m\varphi_y+mG}(\ddbar \varphi_y)^n\\&=C'\int_{\Omega_y}|f|^{2} e^{-(m-1)(\varphi_y-G)}\left| \frac {dz}{dt} \right|^2\\&\leq C'e^{(m-1)\sup G},\end{align*}
where in the last inequality we have used the assumption that $f\in H_{m,y}$.
Hence $|F(x_0)|^{\frac{2}{m}}=|f(x_0)|^{\frac{2}{m}}\leq C'$, by \eqref{Dreg}, $ \varphi_y(x_0) \le C$. The lemma is proved.
 \end{proof}

\begin{proof}[Proof of Theorem \ref{t1} and Theorem \ref{psef}]It is clear that Theorem \ref{t1} follows from Lemma \ref{OT}.
   Set the class \(\{\gamma\} := c_1(L) + \varepsilon\omega\), which is K\"ahler for every positive real number \(\varepsilon > 0\). Then, according to Theorem~\ref{t1}, we obtain that the class $c_1(K_{X/Y} + L + F_{\alpha}) + \varepsilon \omega$ is pseudo-effective. Let $\epsilon\rightarrow 0,$ the limit class $c_1(K_{X/Y} + L + F_{\alpha})$ is also pseudo-effective. Theorem \ref{psef} is also proved.
\end{proof}
\begin{remark}\label{newproof} We remark that Guenancia's Theorem \ref{t:positivity} result is sufficient for the above Theorem as follows: fix a K\"ahler class $[\omega]$ such that $[\omega]-[D]$ is still K\"ahler. Then the conic version of positivity result implies that the class $K_{X/Y}+(1-\epsilon)D+\epsilon(\omega-D)$ is pseudo-effective. Letting $\epsilon\to 0$ will finish the proof.
\end{remark}
Using the same idea,  we slightly generalize Theorem \ref{t1} as follows.

\noindent \textbf{Set-up:} Let $\pi:\mathcal X\to Y$ a holomorphic surjective map between compact K\"ahler manifolds, $D=\sum_{i=1}^r(1-\beta_i) D_i$, where $D_i$ is a reduced divisor with generically simple normal crossings and   $\beta_i\in [0,1]$ such that 
\begin{enumerate}\label{setup1}
\item the cohomology class $c_1(K_{\mathcal X_y}+\sum_{i=1}^r (1-\beta_i){D_i}_{|{\mathcal X_y}})$ is big and nef for every  $y\in Y\setminus\mathcal{S} $, where $\mathcal S$ is the singular locus of $\pi$,
\item further assume that $c_1(K_{\mathcal X_y}+\sum_{i=1}^r (1-\beta_i){D_i}_{|{\mathcal X_y}}-\sum_{i=1}^rd_iD_i|_{\mathcal X_y})$ is K\"ahler, where $d_i>0,1\leq i\leq r$ are fixed positive constants.
\end{enumerate}

\begin{proposition}\label{t1.1}Under the set-up above, then by gluing the fiberwise twisted K\"ahler-Einstein metric with mixed singularities, we obtain a closed positive current $\rho_{mix}$ in 
$ c_1(K_{\mathcal X/Y}+ \sum_{i=1}^r(1-\beta_i) D_i)$.
\end{proposition}

\begin{proof}Under the assumption of Theorem \ref{t1.1}, if we fix a generic fiber, \cite{TY,BG,DT2} proves that, there is a  smooth K\"ahler-Einstein metric on $\mathcal X_y\setminus D_y$ by solving the following equation
\begin{equation}\label{bb}(-\Ric(\omega)+\Ric(D)+\ddbar\phi_{mix})^n=\frac{e^{\phi_{mix}}\omega^n}{\prod_{i=1}^s|s_i|_{h_i}^{2}\prod_{i=s+1}^r|s_i|_{h_i}^{2-2\beta_i}}.\end{equation}

\noindent By assumption, there is a positive constant $\delta_0$ such that, when $0<\delta<\delta_0$,  $$c_1(K_{\mathcal X_y}+D_{|{\mathcal X_y}})-\delta\Theta$$ is a K\"ahler class, where $\Theta:=-\sum_{i=1}^r d_i\ddbar\log(h_i)$. Again we consider the following family of complex Monge-Ampere equations related to conic K\"ahler-Einstein metrics
\begin{equation}\label{AA}(-\Ric(\omega)+\Ric(D)-\delta\Theta+\ddbar\phi_\delta)^n=\frac{e^{\phi_\delta}\omega^n}{\prod_{i=1}^s|s_i|_{h_i}^{2-2\delta d_i}\prod_{i=s+1}^r|s_i|_{h_i}^{2-2\beta_i-2\delta d_i}}.\end{equation}
By \cite{GP}, the above equations admit unique conic solutions depending on $\delta>0$. 


To get an analogue result as Lemma \ref{uni}, we shall get uniform $C^0$ estimate (with barrier) for the K\"ahler potential for a smooth fibration $\mathcal X_0\to Y_0$. So again, we fix a  compact set $K\subset Y_0$. 
\end{proof}

\begin{lemma}\label{c000}Let $\phi_\delta$ be the solution of \eqref{AA},
then for any $\epsilon>0$, there exits a positive constant $C_\epsilon$ such that
$$C\geq \phi_\delta\geq C_\epsilon-\sum_{i=1}^r\log\log^2|s_i|^2_{h_i}+\sum_{i=1}^r\epsilon\log|s_i|^{2d_i}_{h_i}.$$

\end{lemma}
\begin{proof}It is easy to show that there is uniform upper bounds for $\phi_\delta$ and $t\in Y_0$. By the maximum principle and equation \eqref{AA}, we obtain,
$$\phi_\delta+\sum_{i=1}^r(\log\log^2|s_i|^{2}_{h_i}-\epsilon\log|s_i|^{2d_i}_{h_i})\geq\log\frac{(-\Ric(\omega)+\Ric(D)-(\delta+\epsilon)\Theta-\ddbar\sum_{i=1}^r\log\log^2|s_i|^2_{h_i})^n}{\prod_{i=1}^s|s_i|_{h_i}^{2\delta d_i-2}\prod_{i=s+1}^r|s_i|_{h_i}^{2\delta d_i+2\beta_i-2}\prod_{i=1}^r\log^{-2}|s_i|^2_{h_i}\omega^n}.$$
Note that $-\Ric(\omega)+\Ric(D)-\epsilon\Theta$ is a K\"ahler class for fixed $\epsilon$, then the  same argument as in Lemma \ref{uni} show that 
$$\phi_\delta\geq C-\sum_{i=1}^r\log\log^2|s_i|^2_{h_i}+\sum_{i=1}^r\epsilon\log|s_i|^{2d_i}_{h_i},$$
for some $C$ independent of $t\in K$ and $\delta$. The Lemma is proved.
\end{proof}
We also prove a convergence result of $\phi_\delta$ (without passing to subsequence) as $\delta\to 0$, which is parallel to Lemma \ref{t:conetocusp}. This is based on a recent result of \cite{DT2}.
\begin{lemma} \label{c2}Let $\phi_\delta$ be the solution to \eqref{AA}, there exist constants $N,C>0$ such that for all $0<\delta$ sufficiently small,  
\begin{equation}
\sup_{X} \left(  |\sigma_D|^{N}_{h_D} \right)  |\Delta_\theta \phi_{ \delta} | \leq C,
\end{equation}
where $\Delta_\theta$ is the Laplace operator with respect to a fixed K\"ahler metric $\theta$ on $X$. 
In particular, as $\delta\to 0$, 
the conic K\"ahler-Einstein metric $\rho_\delta:=-\Ric(\omega)+\Ric(D)-\delta\Theta+\ddbar\phi_\delta$  converge to the K\"ahler-Einstein metric $\rho_{mix}$ with mixed singularities, in the $\mathscr C_{\rm loc}^{\infty}(X\smallsetminus D)$ topology.
\end{lemma}
\begin{proof} We are going to prove a uniform $C^2$ estimate with barrier. By assumption, we may assume that for some sufficiently small constant $\delta_0$, it holds that
$$-\delta_0\Theta-\Ric(\omega)+\Ric(D)\geq 2\theta.$$
Let $\omega' =  -\Ric(\omega)+\Ric(D)-\delta\Theta+\ddbar\phi_\delta$. Then we consider the quantity $$H= \log tr_\theta (\omega') - B\phi_{ \delta} +B\delta_0\sum_{i=1}^r \log |\sigma_D|^{2d_i}_{h_D}$$
for some large constants $B$ to be determined.

By the $C^0$ estimate of $\phi_{\delta}$, $H$ is bounded above. Standard calculations (cf. \cite[Lemma 3.7]{Gabor}) show that 
\begin{equation}\label{eq:aubin-yau-c2}
\Delta ' \log tr_{\theta}{\omega'} \geq -Ctr_{\omega'}{\theta} - \frac{tr_{\theta}Ric(\omega')}{tr_{\theta}(\omega')}, 
\end{equation} 
where $C$ depends on bisectional curvature of $\theta$.
Direct calculation shows that $$-tr_{\theta}Ric(\omega') \geq \frac{-C}{|\sigma_D|_{h_D}^{2l}},$$ for some constant $C>0$ independent of $\delta$.  Together with \eqref{eq:aubin-yau-c2} and our assumption, we see that 
\begin{align*}
\Delta'H &\geq -Ctr_{\omega'}{\theta}  - \frac{C}{|\sigma_D|_{h_D}^{2l}tr_{\theta}(\omega')} + Btr_{\omega'}(-\Ric(\omega)+\Ric(D)-\delta\Theta -\omega')-B\delta_0tr_{\omega'}\Theta \\
&\geq (B - C)tr_{\omega'}{\theta} - \frac{C}{|\sigma_D|_{h_D}^{2l}tr_{\theta}(\omega')} - Bn\\
&\geq tr_{\omega'}{\theta} - \frac{C}{|\sigma_D|_{h_D}^{2l}tr_{\theta}(\omega')} - Bn \hspace{0.5in} (\text{if } B>>1)\\ 
&\geq (tr_\theta (\omega'))^{\frac{1}{n-1}}(\frac{\theta^n}{\omega'^n})^{\frac{1}{n-1}} - \frac{C}{|\sigma_D|_{h_D}^{2l}tr_{\theta}(\omega')} - Bn\\
&\geq  (tr_\theta (\omega'))^{\frac{1}{n-1}}|\sigma_D|_{h_D}^{2\alpha} - \frac{C}{|\sigma_D|_{h_D}^{2l}tr_{\theta}(\omega')} - Bn,
\end{align*}
for some constant $\alpha$ independent of $\delta$ and $B>>1$ so that $B > C+1$ in line three.

Now assume that $ H$ obtains maximum at a point $p$. Moreover, since $H$ goes to $-\infty$ on $\mathrm{Supp}(D)$, clearly $p\notin \mathrm{Supp}(D)$.  From the maximum principle it follows that at point $p$,
\begin{equation} \frac{C}{|\sigma_D|_{h_D}^{2l}tr_{\theta}(\omega')}+Bn\geq (tr_\theta (\omega'))^{\frac{1}{n-1}}|\sigma_D|_{h_D}^{2\alpha}.\end{equation}
We first assume that $|\sigma_D|^{2l}_{h_D}tr_\theta(\omega')\geq 1$ at $p$, then \begin{equation}(tr_\theta (\omega'))^{\frac{1}{n-1}}|\sigma_D|_{h_D}^{2\alpha}\leq C+Bn.\end{equation}
Notice that the other case is $|\sigma_D|^{2l}_{h_D}tr_\theta(\omega')\leq 1$ at $p$,
if follows that in both cases there is an integer $k$ (depending on $\alpha,l$ and $n$) such that
\begin{equation}|\sigma_D|_{h_D}^{2k}(tr_\theta\omega')(p)\leq C+Bn.\end{equation} Notice that $\phi_\delta\geq \epsilon\log |\sigma_D|^2_{h_D}-C_\epsilon$ for any $\epsilon>0$, so it follows  by choosing a $B>>k$ that 
$H(p)\leq C+Bn$. Now fixing this $B$, we have $H(x)\leq C$ 
for any $x\in X\setminus\mathrm{Supp}(D)$. By the definition of $H$, we have
\begin{align*}
|\sigma_D|^{2B}_{h_D}(tr_{\theta}\omega')(x) \leq C.
\end{align*}  Choosing $N=2B$ will finish the proof for $C^2$ estimate.

By standard higher order estimate, $\phi_\delta$ converge, up to subsequence, to a solution $\phi_{0}$ of \eqref{bb} outside divisor $D$, as $\delta\to 0$. By \cite[Theorem 1.2]{DT2}, $\phi_\delta$ also converge (without passing to subsequence) to $\phi_{mix}$ in the sense of current.  So $\phi_\delta$ converge to  $\phi_{mix}$ smoothly outside $D$
without passing to subsequence. The lemma is proved.
\end{proof}

\begin{proof}[Proof of Proposition \ref{t1.1}:] To prove the proposition, we may use the same argument as Theorem \ref{t1} once  Lemma \ref{c000}, Lemma \ref{c2} and Lemma \ref{l:conedeformation} are proved. \end{proof}

\section{A question of Demailly-Peternell-Schneider: Smooth K\"ahler case}
In this section, we prove Theorem \ref{Demailly}.
\begin{theorem}\label{DEM}(=Theorem \ref{Demailly}) Suppose $\pi:(X,D)\to Y$ is fibration between two compact K\"ahler manifolds. Further suppose that $(X,D)$ is log canonical and $-(K_X+D)$ is nef, then $-K_Y$ is pseudo-effective. 
\end{theorem}
\begin{proof} Set $L:=-(K_X+D)$. Note that $L$ is nef and $K_X+D+L$ is linearly equivalent to zero, then we apply Theorem \ref{psef} to conclude that $K_{X/Y}+D+L=-f^*{K_Y}$ is pseudo-effective. Since $f$ has connected fiber, so $-K_Y$ is also  pseudo-effective.
\end{proof}
\section{Surjectivity of the Albanese map}

In this section, we prove Theorem \ref{t2} by following the argument of \cite{Paun}. Firstly, we recall the definition of Albanese map. Let \(q := h^{1}(X, \mathcal{O}_{X})\) denote the irregularity of \(X\), and define the Albanese torus of \(X\) as
\[
\mathrm{Alb}(X) := H^{0}(X, T_{X}^{*})^{*} / H_{1}(X, \mathbb{Z}).
\]
Recall that the Albanese map \(\alpha_{X} \colon X \to \mathrm{Alb}(X)\) is defined by
\[
\alpha_{X}(x)(\gamma) := \int_{x_{0}}^{x} \gamma,
\]
modulo the subgroup \(H_{1}(X, \mathbb{Z})\); that is, modulo the integrals of \(\gamma\) along loops based at \(x_{0}\).

\begin{theorem} \label{alb}(=Theorem \ref{t2})
Let \(X\) be a compact K\"ahler manifold, and let $D=\sum D_i$ be a simple normal crossing divisor on \(X\), and log canonical divisor \(D_{\alpha} = \sum \alpha_i D_i\) with coefficients \(\alpha_i \in [0,1]\). Assume that the \(\mathbb{R}\)-line bundle \(-(K_{X} + D_{\alpha})\) is nef. Then the Albanese morphism \(\alpha_{X}\colon X \to \mathrm{Alb}(X)\) is surjective.
\end{theorem}


Set \(L := -K_{X} - D_{\alpha}\). Note that while \(L\) is assumed to be nef, the anticanonical bundle \(-K_{X} = L + D_{\alpha}\) is, in general, not nef, but only pseudoeffective.
 Assume by contradiction that the Albanese morphism \(\alpha_X\) is not surjective, and let \(Y \subset \mathrm{Alb}(X)\) be its image. The following criterion will be the key to proving the surjectivity of the Albanese map.

\begin{corollary} \cite[Corollary $10.6$]{Ueno} \label{ueno}
Let $V$ be a complex manifold and let $\alpha \colon V \to A(V)$ be the Albanese mapping of $V$. Then we have
\[ \kappa(\alpha(V)) \geq 0 \]
Moreover, $\kappa(\alpha(V)) = 0$ if and only if the Albanese mapping $\alpha$ is surjective.
\end{corollary}


\begin{proof}[The proof of Theorem \ref{alb}] Firstly, we assume that the generic fiber of the Albanese map is connected. We follow the approach of Păun in \cite{Paun}. We first consider a desingularization \(\widehat{Y}\) of the image \(\alpha_X(X)\). Let \(\overline{X}\) be the fiber product of \(X\) and \(\widehat{Y}\) over \(Y := \alpha_X(X)\). \(\overline{X}\) may be singular, but its singular locus projects onto an analytic subset of \(X\) of codimension at least \(2\). This can be seen, for instance, by considering the rational map \(X \dashrightarrow \widehat{Y}\) obtained by composing the inverse of the resolution map \(\pi_Y\colon \widehat{Y} \to Y\) with the Albanese map \(\alpha_X\). This rational map is defined outside a set of codimension at least \(2\), and \(\overline{X}\) is smooth at each point lying over this regular locus.

Now we take a resolution of singularities \(\widehat{X} \to \overline{X}\) and define the map \(\pi_X \colon \widehat{X} \to X\) as the composition of the desingularization map \(\widehat{X} \to \overline{X}\) with the natural projection \(\overline{X} \to X\). Here, $\pi_X$ is a proper modification of compact K\"ahler manifolds.  Let $E := K_{\widehat{X}} - \pi_X^*(K_X)$ be the exceptional divisors, then the generic fiber of  $p \colon \widehat{X} \to \widehat{Y}$ is disjoint from the support of $E$. In sum, we have the following commutative diagram.

\begin{figure}[H]
    \centering
    \includegraphics[width=0.4\textwidth]{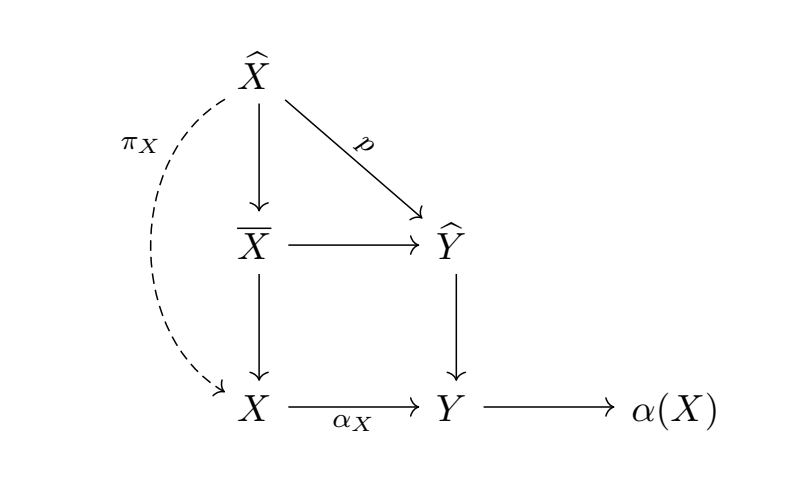}
\end{figure}


Let \( L := -K_X - D_{\alpha} \) be the nef bundle. Then we have
\begin{align*}
    K_{\widehat{X}/\widehat{Y}} + \pi_X^*(-K_X)
    &= K_{\widehat{X}/\widehat{Y}} + \pi_X^*(D_{\alpha} + L) \\
    &= K_{\widehat{X}} - p^*K_{\widehat{Y}} + \pi_X^*(D_{\alpha} + L) \\
    &= \pi_X^*(K_X) + E - p^*K_{\widehat{Y}} + \pi_X^*(D_{\alpha} + L) \\
    &= E - p^*K_{\widehat{Y}}.
\end{align*}
Hence, on the generic fiber \( \widehat{X}_y \) for \( y \in \widehat{Y} \), we get
\begin{align} \label{generic pt}
    \left(K_{\widehat{X}/\widehat{Y}} + \pi_X^*(D_{\alpha} + L)\right)|_{\widehat{X}_y}
    = \left(E - p^*K_{\widehat{Y}}\right)|_{\widehat{X}_y}
    = E|_{\widehat{X}_y}.
\end{align}
Since $E$ is disjoint from the generic fiber of $p$, $E$ is generically $p$ relatively nef.

We now study the positivity of the \( \mathbb{R} \)-line bundle
\[
K_{\widehat{X}/\widehat{Y}} + \pi_X^*D_{\alpha} + \pi_X^*L.
\]
We decompose the pullback \(\pi_X^*D_{\alpha}\) into three (effective) parts:
\[
\pi_X^*D_{\alpha} = D_{\mathrm{ex}} + D_h + D_v,
\]
where \(D_{\mathrm{ex}}\) denotes the \(\pi_X\)-exceptional divisor, which may not be reduced. The divisor \(D_h\) is the \(p\)-horizontal part, i.e., the components that dominate \(\widehat{Y}\), and \(D_v\) is the \(p\)-vertical part, whose components are mapped into proper analytic subsets of \(\widehat{Y}\). By our construction, the divisor \( D_h + D_v \) remains a divisor with simple normal crossing support outside the \( \pi_X \)-exceptional set. We observe that the divisor \( D_h \) fulfills the assumptions of Theorem~\ref{psef}, as each of its irreducible components is generically transverse to the fibers of \( p \) and maps surjectively onto \( \widehat{Y} \) under \( p \). So 
\(
K_{\widehat{X}/\widehat{Y}} + D_h + \pi_X^*L
\)
is pseudoeffective by Theorem~\ref{psef}. Since the divisor $D_v + D_{\mathrm{ex}}$ is effective, it follows that the bundle
\[
K_{\widehat{X}/\widehat{Y}} + \pi_X^*D_{\alpha} + \pi_X^*L
\]
is also pseudo-effective. By the earlier identity, this is equal to
\[
E - p^*(K_{\widehat{Y}}).
\]
Let \(\Lambda\) be a closed positive current representing the cohomology class of \(E - p^*(K_{\widehat{Y}})\). Since the Kodaira dimension \(\kappa(K_{\widehat{Y}}) =\kappa(K_{Y})\geq 1\) by Corollary~\ref{ueno}, we can find two distinct \(\mathbb{Q}\)-effective divisors \(W_1 \not\equiv W_2\) such that
\[
W_1 \sim_{\mathbb{Q}} K_{\widehat{Y}}, \quad W_2 \sim_{\mathbb{Q}} K_{\widehat{Y}}.
\]
In conclusion, we obtain two distinct closed positive currents $T_1$ and $T_2$ in the cohomology class of the exceptional divisor \(E\), namely \(\Lambda + p^*(W_j)\) for \(j = 1, 2\). This leads to a contradiction by \cite[Corollary 4.1]{Paun}. Or one can use the following elementary argument. Since $T_1$ has the same cohomology class as $c_1(E)$, then $T_1=[E]+\ddbar\psi$. Since $T_1$ is positive, $\psi$ is a PSH function on $\hat X\setminus E$. Since $E$ is contracted by $\pi_X$, then Grauert-Remmert Theorem implies that $\psi$ can be extended across $\pi_X(E)$ to a global PSH function on $X$. So $\psi$ is a constant function and $[E]=T_1$. Similarly we have $T_2=[E]$, so $[T_1]=[T_2]$. This is a contradiction.

Now if the Albanese map is not connected, we can pass to the stein factorization of the Albanese map $X\to Y_{stein}\to Alb(X)$. The argument above  shows that the map  $X\to Y_{stein}$ is surjective (here we also use the fact that $\kappa(Y_{stein})\geq\kappa(Alb(X))$ by \cite[Theorem 6.10]{Ueno}). So the Albanese map $X\to  Alb(X)$ is also surjective. The theorem is proved.


\end{proof}

\end{document}